\documentclass[11pt,letterpaper]{amsart}
\usepackage{amsmath}
\usepackage{amssymb}

\usepackage{epsfig}
\usepackage[all, knot]{xy}
\xyoption{arc}

\newtheorem{thm}{Theorem}[section]
\newtheorem*{thm*}{Theorem~1.2}
\newtheorem*{thm**}{Theorem~1.3}
\newtheorem{df}[thm]{Definition}

\newtheorem{cor}[thm]{Corollary}

\newtheorem*{cor*}{Corollary~1.3}

\newtheorem{lem}[thm]{Lemma}
\newtheorem{ex}[thm]{Example}
\newtheorem{rem}[thm]{Remark}
\newcommand{\Pic}{\operatorname{Pic}}

\newcommand{\ord}{\operatorname{ord}}
\newcommand{\Pre}{\operatorname{Preper}}

\newcommand{\Per}{\operatorname{Per}}

\newcommand{\pp}{\mathbb{P}}

\newcommand{\ox}{\mathcal{O}}

\newcommand{\PGL}{\operatorname{PGL}}

\begin{document}

\title[Equidistribution on $K3$ surface]{Equidistribution of periodic points of\\some automorphisms on $K3$~surface}

\author{Chong Gyu Lee}


\date{\today}


\address{Department of Mathematics, University of Illinois at Chicago, Chicago IL 60607, US}

\email{phiel@math.uic.edu}

\maketitle

\begin{abstract}
     We say $(W, \{\phi_1, \cdots, \phi_t\})$ is a polarizable dynamical system of several morphisms if $\phi_i$ are endomorphisms on a projective variety $W$ such that $\bigotimes \phi_i^*L$ is linearly equivalent to $L^{\bigotimes q}$ for some ample line bundle $L$ on $W$ and for some $q>t$. If $q$ is a rational number, then we have the equidistribution of small points of given dynamical system because of Yuan's work \cite{Y}. As its application, we can build a polarizable dynamical system of an automorphism and its inverse on $K3$ surface and show its periodic points are equidistributed.
\end{abstract}

\section{Introduction}

    The study of algebraic dynamics blooms after Northcott proved the arithmetic property of dynamical system of a morphism on projective space. Szpiro, Ullmo \& Zhang \cite{SUZ} started one direction of algebraic dynamics, the equidistribution of small points. After various research of  Bilu \cite{B} on some variety with group structure and of Baker \& Rumely \cite{BR}, Chambert-loir \cite{Ch} and Favre \& Rivera-Letelier \cite{FR} on the equidistribution of dynamical system on dimension $1$, Yuan \cite{Y} proved the general equidistribution theorem: let $\phi$ be a polarizable endomorphisms. Then, we have an ample line bundle $\mathcal{L}$ with semipositive dynamical metric $||\cdot||_\phi$ defined by Zhang \cite{Z1}, then we have the equidistribution of the small points with respect to the height function corresponding to $\overline{\mathcal{L}}= (\mathcal{L},||\cdot||_\phi)$.

    For the dynamical equidistribution, ``polarizable'' condition is very important because it guarantees that we can define a sequence of metric defined on the same line bundle. If $\phi$ is not polarizable, then, metrics ${\phi^k}^*||\cdot||^{\frac{1}{q^k}}$ may be defined on different line bundles for each $k$ so that ``convergence'' of give sequence of metrics doesn't make sense.

    Still, we have hope because of Kawaguchi's idea. He \cite{K0} suggested the polarizable dynamical system of several morphisms:
    \begin{df}
    Let $W$ be a projective variety, let $L$ be an ample line bundle and let $M=\{ \phi_1, \cdots, \phi_t:W \rightarrow W\}$ be a finite set of morphisms. We say that a dynamical system of several morphism $(W, M)$ is \emph{polarizable} if
    \[
    \bigotimes_{i=1}^t \phi_i^*L \sim L^{\otimes q}
    \]
    for some rational number $q>t$.
    \end{df}

    His idea makes a way to study the dynamics of some automorphisms. In general, an automorphisms on projective variety is not polarizable in general. However, we have a good counter part, the inverse map. The existence of the inverse map makes a dynamical system of a automorphism better; all preperiodic points of an automorphisms $\sigma$ is actually periodic, and $\sigma$ and $\sigma^{-1}$ shares that same periodic point. Thus, we can consider a dynamical system of an automorphism $\sigma$ is considered as a dynamical system of $\{\sigma, \sigma^{-1}\}$. Furthermore, a dynamical system of several morphisms $(W, M=\{\phi_1, \cdots \phi_t\})$ actually works with the monoid generated by $M$. If $M$ consists of an automorphism $\sigma$ and its inverse, the monoid $\mathcal{M}$ generated by $M$ is exactly $\{ \sigma^k ~|~ k \in \mathbb{Z}\}$ and hence $\mathcal{M}$-preperiodic points is essentially `$\sigma$-preperiodic or $\sigma^{-1}$-prepriodic points. In Section~2, we have examples of polarizable dynamical systems of an automorphism on $K3$ surfaces.

    The main purpose of this paper is to confirm that we have the dynamical equidistribution for dynamical systems of several morphisms and to apply this result on some automorphisms on $K3$ surface. In section~3, we will combine Kawaguchi's and Yuan's results to prove the equidistribution of small points:
     \begin{thm*}
        Let $W$ be a projective variety of dimension $n$ over a number field $K$, let $\mathcal{L}$ be an ample line bundle and let $M= \{ \phi_1, \cdots, \phi_t\}$ be a finite set of endomorphisms on $W$. Suppose that $(W,M)$ is a polarizable with some integer $q>t$ and $\{ x_m\}$ be a generic and small sequence. Then, the a sequence of probability measure on the Galois orbit of $x_m$ weakly converges to the dynamical measure at every place $v$:
        \[
        \dfrac{1}{\deg x_m} \sum_{y\in \Gamma_{x_m}} \delta_y \rightarrow \mu_{M,v}
        \]
        where $\Gamma_{x_m}$ is the Galois orbit or $x_m$ and $\mu_{M,v} = \dfrac{c_1(\mathcal{L})^n_v}{\deg_L W}$ is the dynamical probability measure of the dynamical system $(W,M)$ on the analytic space $W_{K_v}^{an}$.
    \end{thm*}

    In Section~5, we will show that we can find a generic and small sequence of periodic points. Thus, we can find some properties of the set pr periodic points of some automorphisms on $K3$ surfaces.

    \begin{thm**}
        Let $W$ be a projective variety defined over a number field $K$, let $M = \{\sigma, \sigma^{-1}\}$ be an automorphism and its inverse on $W$. Suppose that $(W,M)$ is polarizable with some integer $q>2$. Then, $\Per(\sigma)$ is Zariski dense.
    \end{thm**}

\par\noindent\emph{Acknowledgements}.\enspace
 The author would like to thank Xinyi Yuan for helpful discussions and comments for paper, thank Joseph H. Silverman for suggesting ideal for proof of Theorem~\ref{dense}. Also thanks to Shu Kawaguchi and Jordan Ellenberg for useful comments.

\section{Polarizable dynamical systems of automorphisms on $K3$ surfaces}

    We have lots of interesting examples of the polarizable dynamical system of several morphisms on $K3$ surface.

    \subsection{$K3$ surface with two involutions, I}
    The space of $K3$ surfaces is a $19$-dimensional object up to isomorphism. And, a family of $K3$ surface in $\pp^2\times \pp^2$ defined by an intersection of hypersurfaces of bidegree $(1,1)$ and $(2,2)$ is $18$-parameter family of isomorphism classes of nonsingular surfaces. For details of such $K3$ surfaces, refer \cite[\S 7.4]{S2}.

    \begin{ex}\label{K3I}
        Let $S =\pp^2\times \pp^2$ be a $K3$ surface defined by an intersection of hypersurfaces of bidegree $(1,1)$ and $(2,2)$ with two involutions $\iota_1, \iota_2$, let $\pi_i$ be the projection map onto $i$-th component and let $L_i = \pi_i^*\ox_{\pp^2}(1)$.
        Then, we have
        \[
        \iota_i^*L_i = L_i,  \quad \iota_i^*L_j = L_i^{\otimes 4} \otimes L_j^{\otimes-1}
        \]
        and hence
        \[
        \iota_1^*L \otimes \iota_2^*L = L^{\otimes 4}
        \]
        where $L =L_1 \otimes L_2$ is an ample line bundle. Therefore, $(S, \{ \iota_1, \iota_2 \})$ is a polarizable dynamical system.
    \end{ex}

    \begin{ex}\label{K3IB}
        Let $S$ be the $K3$ surface defined on Example~\ref{K3I}. Define $\sigma_1 = \iota_2 \circ \iota_1$ and $\sigma_2 = \iota_1 \circ \iota_2 = \sigma_1^{-1}$. Then,
        $(S, \{\sigma_1, \sigma_2\})$ is a polarizable dynamical system: we have
         \[
        \sigma_i^*L_i = L_i^{\otimes -1} \otimes L_j^{\otimes 4}, \quad  \sigma_i^*L_j = L_i^{\otimes -4} \otimes L_j^{\otimes 15}
        \]
        and
        \[
        \sigma_1^*L \otimes \sigma_2^*L = L^{\otimes 14}.
        \]
    \end{ex}

    \subsection{$K3$ surface with two involutions, II}

    There is another way of defining other automorphisms on $K3$ surface in $\pp^2 \times \pp^2$. With same method, we can calculate the
    number of parameters of defining equations
     \[
     \sum_{0\geq i \geq j \geq 2}\sum_{0\geq k \geq 2} A_{ijk}  x_{i} x_j y_k  \quad \sum_{0\geq l \geq 2}\sum_{0\geq m \geq n \geq 2}
     B_{mmn}  x_{l} y_m y_n
     \]
     is $(18-1)+(18-1)$ and the dimension of $\PGL_3$, the isometry group of each $\pp^2$, is 8. Hence the dimension of the family of $K3$ such surfaces is $18$ again.

    \begin{ex}\label{K3II}
        Let $S =\pp^2 \times \pp^2$ be a $K3$ surface generated by intersecting two hypersurfaces of bidegree $(1,2)$ and $(2,1)$ with two involutions $\iota_1, \iota_2$. Let $\pi_i$ be the projection map onto $i$-th component and $L_i = \pi_i^*\ox_{\pp^2}(1)$.
        Then, we have
        \[
        \iota_i^*L_i = L_i, \quad \iota_i^*L_j = L_i^{\otimes -1} \otimes L_j^{\otimes 5}.
        \]
        Since $L =L_1 \otimes L_2$ is ample, we get
        \[
        \iota_1^*L \otimes \iota_2^*L = L^{\otimes 5}
        \]
        and hence get a polarizable dynamical system.
    \end{ex}

    \begin{ex}\label{K3IIB}
         Let $S$ be the $K3$ surface defined on Example~\ref{K3II}. Define $\sigma_1 = \iota_2 \circ \iota_1$ and $\sigma_2 = \iota_1 \circ \iota_2=\sigma_1^{-1}$. Then,
        $(S, \sigma_1, \sigma_2)$ is a rational polarizable dynamical system: we have
         \[
        \sigma_i^*L_i = L_i^{\otimes -1} \otimes L_j^{\otimes 5}, \quad  \sigma_i^*L_j = L_i^{\otimes -5} \otimes L_j^{\otimes 24}
        \]
        and
        \[
        \sigma_1^*L \otimes \sigma_2^*L  = L^{\otimes 23}
        \]
        for all $L \in \langle L_1, L_2 \rangle$.
    \end{ex}

    \subsection{$K3$ surface with three involutions}

    If we define a $K3$ surface in $\pp_1 \times \pp^1 \times \pp^1$, then N\'{e}ron-Severi group is of rank $3$. Thus we expect that the dimension of the family of $K3$ surface is reduced by $1$; the
    number of parameters for defining equations
     \[
     \sum_{0\geq i \geq j \geq 2}\sum_{0\geq k \geq l\geq 2}\sum_{0\geq m \geq n \geq 2} A_{ijklmmn}I x_{i}xji y_k y_l z_m z_n
     \]
     is $27-1$ and the dimension of $\PGL_2$, the isometry group of each $\pp^1$, is 3. Hence
      the dimension of the family of $K3$ such surfaces is $26 - 3 - 3 - 3 = 17$.
    
    \begin{ex}\label{K3III}
        Let $S = \pp^1 \times \pp^1 \times \pp^1$ be $K3$ surface, a hypersurface of bidegree $(2,2,2)$ with three involutions $\iota_1, \iota_2, \iota_3$. Let $\pi_i$ be the projection map onto $i$-th component and $L_i = \pi_i^*\ox_{\pp^1}(1)$. Then, we have
        \[
        \iota_i^*L_j = L_j~\text{for}~i\neq j, \quad \iota_i^*L_i = L_i^{\otimes -1} \otimes  L_j^{\otimes2} \otimes L_k^{\otimes2}.
        \]
        Hence, let $L =L_1 \otimes L_2 \otimes L_3$ and get
        \[
        \iota_1^*L \otimes \iota_2^*L \otimes \iota_3^*L = L^{\otimes 5}.
        \]
        So, $(S, \{\iota_1, \iota_2, \iota_3\})$ is polarizable.
    \end{ex}

    \begin{ex}\label{K3IIIB}
         Let $S$ be the $K3$ surface defined on Example~\ref{K3III}. Consider $\tau_1 = \iota_3\circ \iota_2 \circ \iota_1$, $\tau_2 = \tau_1^{-1}$. Then,
        \[
        \begin{array}{c@{~=~}l@{~=~}l@{~=~}l}
        \tau_1^*L_1 & \iota_3^* \iota_2 (L_1^{\otimes -1} \otimes  L_2^{\otimes 2} \otimes L_3^{\otimes 2}) &
        \iota_3^*(L_1^{\otimes 3} \otimes  L_2^{\otimes -2} \otimes L_3^{\otimes 6}) &
        L_1^{\otimes 15} \otimes  L_2^{\otimes 10} \otimes L_3^{\otimes -6} \\
        \tau_1^*L_2 & \iota_3^*\iota_2^* (L_2) &  \iota_3^*(L_1^{\otimes 2} \otimes  L_2^{\otimes -1} \otimes L_3^{\otimes 2}) &
        L_1^{\otimes 6} \otimes  L_2^{\otimes 3} \otimes L_3^{\otimes -2} \\
        \tau_1^*L_3 & \iota_3^*\iota_2^* (L_3) & \iota_3^* L_3 & L_1^{\otimes 2} \otimes  L_2^{\otimes 2} \otimes L_3^{\otimes -1}
        \end{array}
        \]
        \[
        \begin{array}{c@{~=~}l@{~=~}l@{~=~}l}
        \tau_2^*L_1 & \iota_1^*\iota_2^* (L_1) & \iota_1^* L_1 & L_1^{\otimes -1} \otimes  L_2^{\otimes 2} \otimes L_3^{\otimes 2}\\
        \tau_2^*L_2 & \iota_1^*\iota_2^* (L_2) &  \iota_1^*(L_1^{\otimes 2} \otimes  L_2^{\otimes -1} \otimes L_3^{\otimes 2}) &
        L_1^{\otimes -2} \otimes  L_2^{\otimes 3} \otimes L_3^{\otimes 6} \\
        \tau_2^*L_3 & \iota_1^* \iota_2 (L_1^{\otimes 2} \otimes  L_2^{\otimes 2} \otimes L_3^{\otimes -1}) &
        \iota_1^*(L_1^{\otimes 3} \otimes  L_2^{\otimes -2} \otimes L_3^{\otimes 6}) & L_1^{\otimes -6} \otimes  L_2^{\otimes 10} \otimes L_3^{\otimes 15} \\
        \end{array}
        \]

        \begin{eqnarray*}
        [\tau_1^*\otimes \tau_2^*](L_1^{\otimes a} \otimes  L_2^{\otimes b} \otimes L_3^{\otimes c})
         &=& (L_1^{\otimes 15a+6b+2c} \otimes  L_2^{\otimes 10a+3b+2c} \otimes L_3^{\otimes -6a-2b-c}) \\
         & & \otimes (L_1^{\otimes -a-2b-6c} \otimes  L_2^{\otimes  2a+3b+ 10c } \otimes L_3^{\otimes 2a+6b+15c}) \\
         &=&          L_1^{\otimes 14a+4b-4c} \otimes  L_2^{\otimes 12a+6b+12c} \otimes L_3^{\otimes -4a+4b+14c}
        \end{eqnarray*}
        Therefore, let $L = L_1 \otimes L_2^{\otimes 2} \otimes L_3$. Then,
        \[
        \tau_1^*L \otimes \tau_2^*L \sim L^{\otimes 18}
        \]
        and hence
        $(S, \tau_1, \tau_1^{-1})$ is polarizable. More precisely, let $L_{\alpha, \beta} = (L_1 \otimes L_2)^{\otimes \alpha} \otimes (L_1^{-1} \otimes L_3)^{\otimes \beta}$, then
        \[
        \tau_1^*L_{\alpha, \beta} \otimes \tau_2^*L_{\alpha, \beta} \sim L_{\alpha, \beta}^{\otimes 18}.
        \]

        Similarly, automorphisms $\tau' = \iota_1\circ \iota_3 \circ \iota_2$, $\tau''= \iota_2 \circ \iota_1 \circ \iota_3$ with their inverses will generate polarizable dynamical systems respectively.
    \end{ex}

     \begin{ex}\label{ce}
    Consider the following case; let $S$ be a $K3$ surface defined on Example~\ref{K3III}. Define $\sigma_1 = \iota_2 \circ \iota_1$ and $\sigma_2 = \iota_1 \circ \iota_2 = \sigma_1^{-1}$. Then,
        $(S, \sigma_1, \sigma_2)$ is a polarizable dynamical system:
        \[
        \begin{array}{c@{~=~}l@{~=~}l}
        \sigma_1^*L_1 & \iota_2^* (L_1^{\otimes -1} \otimes  L_2^{\otimes 2} \otimes L_3^{\otimes 2}) &
        L_1^{\otimes 3} \otimes  L_2^{\otimes -2} \otimes L_3^{\otimes 6} \\
        \sigma_1^*L_2 & \iota_2^* (L_2) &  L_1^{\otimes 2} \otimes  L_2^{\otimes -1} \otimes L_3^{\otimes 2} \\
        \sigma_1^*L_3 & \iota_2^* (L_3) & L_3
        \end{array}
        \]
        Therefore,
        \begin{eqnarray*}
        [\sigma_1^*\otimes \sigma_2^*](L_1^{\otimes a} \otimes  L_2^{\otimes b} \otimes L_3^{\otimes c})
         &=& (L_1^{\otimes 3a+2b} \otimes  L_2^{\otimes -2a-b} \otimes L_3^{\otimes 6a+2b+c}) \\
         & & \otimes (L_1^{\otimes -a-2b} \otimes  L_2^{\otimes 2a+3b } \otimes L_3^{\otimes 2a+6b+c}) \\
         &=&          L_1^{\otimes 2a} \otimes  L_2^{\otimes 2b} \otimes L_3^{\otimes 8a+8b+2c}
        \end{eqnarray*}
        Therefore, $L_3$ is the only combination of $L_1, L_2$ and $L_3$ which makes linear equivalence;
        \[
        \sigma_1^*L_3 \otimes \sigma_2^* L_3 \sim L_3^{\otimes 2}
        \]
        and hence $(S, \{\sigma_1, \sigma_2\})$ is not a polarizable dynamical system in Kawaguchi's sense.
        Similarly,
        $(S, \{\iota_1\circ \iota_3, \iota_3\circ \iota_1\})$,  $(S, \{\iota_2\circ \iota_3, \iota_3\circ \iota_2\})$
        are not polarizable.
    \end{ex}

    \subsection{$K3$ surface with three involutions, of the Picard number $4$.}
    \begin{ex}
    Let $S = \pp^1 \times \pp^1 \times \pp^1$ be $K3$ surface, a hypersurface of bidegree $(2,2,2)$ of the Picard number $4$. Then, we have another involution $\iota_4$ of order $2$, which is a group inverse at each elliptic curve fibers of $S$. Then, $\Pic(S) = \langle L_1, L_2, L_3, L_4 \rangle$ where $L_4$ corresponds to $-2$-curve class containing $(x,0,0)$. (See \cite{BM} for details.) Define an automorphisms $\tau = \iota_1 \circ \iota_2 \circ \iota_4$ and $\tau^{-1} = \iota_4 \circ \iota_2 \circ \iota_1$. Then,
    \[
    \iota_4^*L_1 = L_1, \iota_4^*L_4 = L_4, \quad  \iota_4^*L_j = L_1^{\otimes 8} \otimes L_i^{-1} \otimes L_4.
    \]
    Therefore,
    \[
    \tau^*L \otimes {\tau^{-1}}^* L = L^{\otimes 30}
    \]
    if
    \[L =  (L_1^{\otimes 4} \otimes  L_2^{\otimes  4} \otimes L_3^{\otimes 3})^{\otimes \alpha} \otimes
     (L_1^{\otimes 8} \otimes  L_2^{\otimes  2} \otimes L_4^{\otimes 3})^{\otimes \beta}
     \]
     for some $\alpha, \beta$. Similarly, $\tau' = \iota_1 \circ \iota_3 \circ \iota_4$ with its inverse
     will generate a polarizable dynamical system. But, $\zeta = \iota_2 \circ \iota_3 \circ \iota_4$ or $\zeta' = \iota_3 \circ \iota_2 \circ \iota_4$ only polarized by $q=2$ or $q=-2$. Also, $\eta = \iota_1 \circ \iota_4\circ \iota_2$, $\eta' = \iota_1 \circ \iota_4 \circ \iota_2$ is not polarized by any ample line bundles.
    \end{ex}

\section{Dynamical Equidistribution on polarizable dynamical systems of several morphisms}

    The equidistribution of small points for polarizable morphisms is almost proved. Kawaguchi proved that the dynamical system of several morphisms generates the dynamical adelic metric which is semipositive.  Thus, the equidistribution of small points of a polarizable dynamical system is an easy consequence of Yuan's results. In this section, we will briefly check Kawaguchi's and Yuan's results to confirm the equidistribution theory for dynamical system of several morphisms.

    \begin{df}
        Let $W$ be a projective variety defined over a number field $K$, let $\overline{L} = (L,||\cdot||)$ be an adelic ample line bundle  on $W$, an ample line bundle $L$ with a semipositive adelic metric$||\cdot||$. Then we define \emph{a height of subvarieties corresponding to $\overline{L}$} to be
        \[
        {h}_{\overline{L}} (Y) := \dfrac{c_1({\overline{L}})^{d+1}}{(d+1)\ord_{\overline{L}} Y},
        \]
        where $Y$ is a subvariety of $W$ of dimension $d$ and $c_1$ is the curvature form.
    \end{df}

    \begin{df}
        Let $W$ be a projective variety, let $\overline{L}$ be an adelic ample line bundle and let ${h}_{\overline{L}}$ be the height function for closed subvarieties of $W$ corresponding line bundle $\overline{L}$. Suppose $\{x_m\}$ is a sequence of points. Then we say
        \emph{$\{x_m\}$ is generic} if any infinite subsequence is not contained in a closed subvariety. We say\emph{$\{x_m\}$ is small}
        if ${h}_{\overline{L}}(x_m)$ converges to ${h}_{\overline{L}}(W)$.
    \end{df}

    \begin{thm}[{\cite[Theorem 3.2]{Y}}]\label{yuan}
        Suppose that $W$ is a projective variety of dimension $n$ over a number field $K$, and $\overline{L}$ is a metrized line bundle over $W$with semipositive adelic metric. Let $\{x_m\}$ be an infinite sequence of closed point in $W$ which is generic and small with respect to $h_{\overline{L}}$. Then for any place $v$ of $K$, the Galois orbit of sequence $\{x_m\}$ are equidistributed in the analytic space $W^{an}_{K_{v}}$ with respect to the canonical measure $d\mu_{v} = c_1(\overline{L})^n_v / \deg_{\overline{L}} W$:
        \[
        \dfrac{1}{\deg x_m} \sum_{y \in \Gamma x_m} \delta_y ~\text{weakly converges to}~ d\mu_{v}.
        \]
    \end{thm}

    \begin{rem}
        In Theorem~\ref{yuan}, we should assume that $L$ is $\mathbb{Q}$-divisor. Actually we will use the integral model $(\mathcal{W}, \mathcal{L})$ of $(W,L^e)$ where  $\mathcal{L}$ on the generic fiber $\mathcal{W}_{\mathbb{Q}}$ is $L^e$. Thus, $L \in \Pic(W)\otimes \mathbb{Q}$. Using $\mathcal{L}$ instead of $L$ implies this fact.
    \end{rem}

    \begin{thm}[{\cite[Theorem A,B]{K0}}]\label{kawa}
        Let $W$ be a projective variety defined over a number field $K$, Let $L$ be an ample line bundle on $W$, and let $M = \{\phi_1, \cdots \phi_t\}$ be a set of endomorphisms on $W$. Suppose that $(W,M)$ is polarizable with respect to $L$:
        \[
        \bigotimes_{i=1}^t \phi_i^*L = L^{\otimes q}
        \]
        where $q>t$. Then,
            \begin{enumerate}
                \item There is a unique continuous metric $||\cdot ||_M$, called the admissible metric on $L$ with $||\cdot ||^q_M = \tau^*(\phi_1^*|| \cdot ||^q_M \cdots \phi_t^* ||\cdot ||^q_M)$ where $\tau : L^{\otimes q} \rightarrow \bigotimes \phi_i^*L$ is an isomorphism.
                \item Let $\overline{L}= (L, ||\cdot||_M)$ be the line bundle with admissible metric. Then, there exists a unique real-valued function
                    \[
                        \widehat{h}_{\overline{L}} : W(\overline{K}) \rightarrow \mathbb{R}
                    \]
                with the following properties:
                    \begin{enumerate}
                        \item $\widehat{h}_{\overline{L}}$ is a Weil height corresponding to $L$.
                        \item $\displaystyle \sum_{i=1}^k \widehat{h}_{\overline{L}} \bigl( \phi_i(x)\bigr) = q \cdot \widehat{h}_{\overline{L}} (x)$ for all $x \in W(\overline{K})$.
                    \end{enumerate}
                \item $\widehat{h}_{\overline{L}} \geq 0$ for all $x\in W(\overline{K})$.
            \end{enumerate}
    \end{thm}

    \begin{rem}
        The condition $q>t$ is necessary because the number of $N$-combinations of $\phi_i$'s in $M$ is $t^N$ while the growth rate of height
        is $q^N$. More precisely, the canonical height, if exists, defined by several morphism is of the form
        \[
        \lim_{N\rightarrow \infty} \dfrac{1}{q^N} \sum_{F \in \mathcal{M}_N} h_{\overline{L}} \bigl( F(P) \bigr).
        \]
        Therefore, if $q\leq k$, then it may not shrink at prepriodic points. For example, if $q=k$ and $P$ is a common fixed point of $M$ with nontrivial height value, then
        \[
        \lim_{N\rightarrow \infty} \dfrac{1}{q^N} \sum_{F \in \mathcal{M}_N} h_{\overline{L}}\bigl( F(P) \bigr) = h_{\overline{L}}(P)
        \]
        so that $P$ may not be a root of the canonical height. Thus, even we can build a semipositive metric, it is not compatible with the original dynamical system.
    \end{rem}

    Now, combining previous two Lemmas, we get equidistribution of small points for the polarizable dynamical system of several morphisms:

     \begin{thm*}
        Let $W$ be a projective variety of dimension $n$ over a number field $K$, let $L$ be an ample line bundle and let $M= \{ \phi_1, \cdots, \phi_t\}$ be a finite set of endomorphisms on $W$. Suppose that $(W,M)$ is a polarizable with some integer $q>t$ and $\{ x_m\}$ be a generic and small sequence. Then, the a sequence of probability measure on the Galois orbit of $x_m$ weakly converges to the dynamical measure at every place $v$:
        \[
        \dfrac{1}{\deg x_m} \sum_{y\in \Gamma_{x_m}} \delta_y \rightarrow \mu_{M,v}
        \]
        where $\Gamma_{x_m}$ is the Galois orbit of $x_m$ and $\mu_{M,v} = \dfrac{c_1(L)^n_v}{\deg_L W}$ is the probability $M$-invariant measure on the analytic space $W_{K_v}^{an}$.
    \end{thm*}
    \begin{proof}
        Theorem~\ref{kawa} says that we have the dynamical adelic metric on $L$. Such metric is semipositive because it is defeind by the limit of positive metrics.
         Since we assume that $q$ is rational number, we may assume that $\mathcal{L}$ is $\mathbb{Q}$ divisor and hence we can apply Theorem~\ref{yuan}, to get the desired result.
    \end{proof}

\section{Periodic points of automorphisms}

       In previous section, we have the equidistribution of small points for the polarizable dynamical systems of several morphisms. To show the periodic points are equidistributed, we can make the generic sequence and hence we should show that the set of periodic points is Zariski dense. We can prove it in two ways: Fakhruddin \cite{Fa} introduce algebro-geometric proof by Poonen. I will introduce the arithmetic proof.

    \begin{lem}\label{induction}
         Let $\sigma :W \rightarrow W$ be an automorphism on a surface $W$ such that $(W, \{ \sigma, \sigma^{-1}\})$ is polarizable. Then, $(W, \{ \sigma^m, \sigma^{-m}\})$ is also polarizable for all $m \in \mathbb{Z}$.
    \end{lem}
    \begin{proof}
            Let
        \[
        \sigma^*L \otimes {\sigma^{-1}}^*L = L^{\otimes q}\quad \text{where}~ q>2.
        \]
        Then,
        \[
        {\sigma^2}^*L \otimes {\sigma^{-2}}^*L = L^{\otimes q^2-2}
        \]
        By induction, suppose
        \[
        L_m = {\sigma^l}^*m \otimes {\sigma^{-l}}^*L = L^{\otimes q_m}\quad \text{where}~ q_m>2
        \]
        and $q_{m} - q_{m-1}>2$ holds for $m=m_0-1, m_0$.  Then,
        Then,
        \[
        {\sigma}^*L_{m_0} \otimes {\sigma^{-1}}^*L_{m_0}
        = L^{\otimes q \cdot q_{m_0}}.
        \]
        On the other hand,
        \[
        {\sigma}^*L_{m_0} \otimes {\sigma^{-1}}^*L_{m_0}
        = {\sigma^{{m_0}+1}}^*L \otimes {\sigma^{{m_0}-1}}^*L  \otimes {\sigma^{-{m_0}+1}}^*L \otimes {\sigma^{-{m_0}-1}}^*L
        \]
        and hence
        \[
        {\sigma^{{m_0}+1}}^*L \otimes  {\sigma^{-{m_0}-1}}^*L = L^{\otimes q\cdot q_{m_0} - q_{{m_0}-1}}.
        \]
        Therefore,
        \[
        {\sigma^{{m_0}+1}}^*L \otimes  {\sigma^{-{m_0}-1}}^*L = L^{\otimes q_{m_0}}
        \]
        where $q_{m_0} = q\cdot q_{{m_0}+1} - q_{{m_0}} \geq 2q_{{m_0}} - q_{{m_0}-1}>2$. Moreover, $q_{{m_0}+1} - q_{{m_0}} = (q-1)q_{{m_0}} - q_{{m_0}-1} > q_{m_0} - q_{{m_0}-1} >2$.
    \end{proof}

    \begin{thm**}\label{dense}
        Let $W$ be a projective variety defined over a number field $K$, let $M = \{\phi, \phi^{-1}\}$ be an automorphism and its inverse on $W$. Suppose that $(W,M)$ is polarizable with some integer $q>2$. Then, $\Per(\phi)$ is Zariski dense.
    \end{thm**}
    \begin{proof}
        Suppose that the Zariski closure of $\mathcal{S}$ is $\mathbf{C} = \{ C_1, \cdots , C_r\}$ where $C_i$ are irreducible curves.
        Then, $\sigma(C_i)$ is still in $\mathcal{C}$ and hence $\phi$ works as maps $\mathbf{C}$ into itself. Thus, there is a curve fixed by $\sigma^{m}$ for some $m>0$. Without loss of generality, $\sigma^m(C_1) = C_1$. Furthermore, since $\sigma$ is an automorphism,
        $\sigma^m_*C_1 = C_1$. Similarly, $\sigma^{-m}_*C_1 = C_1$.  Then,
        \[
        L \cdot C_1 = L^d \cdot \sigma^m_* C_1 = {\sigma^m}^*L \cdot C_1  \quad L \cdot C_1 = L \cdot \sigma^m_* C_1 = {\sigma^m}^*L \cdot C_1
        \]
        and hence
        \[
        L^{\otimes 2}  \cdot C_1 = {\sigma^m}^*L \cdot C_1  +  {\sigma^m}^*L \cdot C_1 =  \left({\sigma^m}^*L \otimes  {\sigma^m}^*L \right)\cdot C_1.
        \]

        On the other hand, by Lemma~\ref{induction}, $(W, \{\sigma^m, \sigma^{m-}\})$ is also polarizable. Thus,
        \[
        L^{\otimes 2}  \cdot C_1 = L^{\otimes q_m}  \cdot C_1 \quad \text{where}~q_m >2
        \]
        and hence $L  \cdot C_1 = 0$, which contradicts to the assumption that $L$ is ample.
    \end{proof}

    \begin{cor}
        Let $\sigma :W \rightarrow W$ be an automorphism on a projective surface $W$ such that $(W, \{ \sigma, \sigma^{-1}\})$ is polarizable
        with respect to an ample line bundle $L$. Then, a generic sequence $\{ x_m\} \subset \Per(\sigma)$ is small.
    \end{cor}
    \begin{proof}
        We know that $\widehat{h}_{\overline{L}}(x_m)=0$ if $x_m \in \Per(\sigma)$. Thus, we only have to show that
        $\widehat{h}_{\overline{L}}(W)=0$.

        \cite[Theorem 1.10]{Z1} says that
        \[
        e_1(L)   \geq
        \widehat{h}_{\overline{L}}(W) \geq \dfrac{1}{n} \sum_{i=1}^n e_i (\overline{L})
        \]
        where
        \[
        e_i(\overline{L})
        = \sup_{\substack{Y \subsetneq W \\ \operatorname{Codim} Y = i}} \inf_{x \in W \setminus Y} \widehat{h}_{\overline{L}} (x).
        \]

        By Theorem~\ref{dense}, $\Per(\sigma)$ is Zariski Dense in $W$ and hence $e_i(\overline{L}) =0$ for all $i = 1 \cdots n$ and hence $\widehat{h}_{\overline{L}}(W) =0$.
    \end{proof}

    Now, Theorem~\ref{dense} say that we can build a generic and small sequence of periodic points. Therefore, we can prove the equidistribution of periodic point:

    \begin{cor}
        Let $W$ be a projective variety defined over a number field $K$, let $M = \{\sigma, \sigma^{-1}\}$ be an automorphism and its inverse on $W$. Suppose that $(W,M)$ is polarizable with some integer $q>2$. Then, $\Per(\sigma)$ is equidistributed.
    \end{cor}
    \begin{proof}
           It is an easy consequence of Theorem~1.2. and Theorem~1.3: by Theorem~1.3, we can build a generic and small sequence $\{x_m\}$ in $\Pre(\sigma)$. And, the sequence of probability measure on Galois orbit of $x_m$ weakly converges to the dynamical measure by Theorem~1.2.
    \end{proof}

\end{document}